\documentclass[a4paper,11pt]{amsart}
\usepackage[latin1]{inputenc}
\usepackage{amsmath}
\usepackage{amsfonts}
\usepackage{amssymb}
\usepackage{amsthm}
\usepackage{mathrsfs}

\newcommand{\R}{{\mathbb{R}}}

\newcommand{\Z}{{\mathbb{Z}}}

\def\vecx{{\text{\boldmath$x$}}}

\def\vecu{{\text{\boldmath$u$}}}
\def\vecv{{\text{\boldmath$v$}}}
\def\vecw{{\text{\boldmath$w$}}}
\def\vecm{{\text{\boldmath$m$}}}

\def\vecell{{\text{\boldmath$\ell$}}}

\def\vecp{{\text{\boldmath$p$}}}

\def\vec0{{\text{\boldmath$0$}}}

\newcommand{\ve}{\varepsilon}

\newcommand{\sfrac}[2]{{\textstyle \frac {#1}{#2}}}

\newcommand{\SL}{\mathrm{SL}}

\newtheorem{thm}{Theorem}[section]
\newtheorem{lem}[thm]{Lemma}
\newtheorem{prop}[thm]{Proposition}
\newtheorem{cor}[thm]{Corollary}

\theoremstyle{remark}
\newtheorem{remark}[thm]{Remark}

\numberwithin{equation}{section}

\begin{document}
\title[On the distribution of angles in a random lattice]{On the distribution of angles between the $N$ shortest vectors in a random lattice}
\author{Anders S\"odergren}
\address{Department of Mathematics, Uppsala University, Box 480,\newline
\rule[0ex]{0ex}{0ex}\hspace{8pt} SE-75106 Uppsala, Sweden\newline
\rule[0ex]{0ex}{0ex}\hspace{8pt} {\tt sodergren@math.uu.se}} 
\date{\today}
\thanks{The research has been supported by the Swedish Research Council, Research Grant 621-2007-6352.}

\maketitle

\begin{abstract}
We determine the joint distribution of the lengths of, and angles between, the $N$ shortest lattice vectors in a random $n$-dimensional lattice as $n\to\infty$. Moreover we interpret the result in terms of eigenvalues and eigenfunctions of the Laplacian on flat tori. Finally we discuss the limit distribution of any finite number of successive minima of a random $n$-dimensional lattice as $n\to\infty$. 
\end{abstract}

\section{Introduction}

For $n\in\Z_{\geq1}$ let $X_n$ denote the space of $n$-dimensional lattices of covolume $1$. We realize $X_n$ as the homogeneous space $\SL(n,\Z)\backslash \SL(n,\R)$, where $\SL(n,\Z)g$ corresponds to the lattice $\Z^ng\subset\R^n$. We further let $\mu_n$ denote the Haar measure on $\SL(n,\R)$, normalized so that it represents the unique right $\SL(n,\R)$-invariant probability measure on the space $X_n$.

Given a lattice $L\in X_n$, we order its non-zero vectors by increasing lengths as $\pm\vecv_1,\pm\vecv_2,\pm\vecv_3,\ldots$. The first several vectors in this list are important objects attached to $L$. Indeed, from knowledge of a relatively short vector  in any given lattice $L$, one can obtain integer solutions to a variety of different problems, including that of factoring polynomials with rational coefficients; cf., e.g., \cite{LLL}, \cite{LLL2}. Note also that the shortest non-zero vector of $L$, i.e.\ $\vecv_1$, determines the density of the sphere packing based on $L$, so that finding the lattice $L\in X_n$ which maximizes the length of $\vecv_1$ is equivalent to the classical problem of finding the maximal density of a lattice sphere packing in $\R^n$.

Our purpose in the present paper is to study the distribution of lengths and relative positions of the vectors $\pm\vecv_1,\ldots,\pm\vecv_N$ for a random lattice in large dimension, i.e.\ an $n$-dimensional lattice chosen according to the measure $\mu_n$ on $X_n$, for $N$ fixed and $n\to\infty$.

There exist in the literature many different notions of "random" lattices in $\R^n$, cf., e.g., \cite{NDB}. However, the probability measure $\mu_n$ used in the present paper is the natural one when viewing the space of lattices, $X_n$, as a homogeneous space. We also note that, in recent years, probabilities defined in terms of $\mu_n$-random lattices have appeared in a number of applications in number theory and mathematical physics; cf.\  \cite{ElkMM}, \cite{jens1}, \cite{JMAS1}, \cite{JMAS2}, \cite{jens2},  \cite{ASAV}. 

In a previous paper \cite{jag} we study, for large $n$, the distribution of lengths of lattice vectors in a random lattice $L\in X_n$. With $\pm\vecv_1,\pm\vecv_2,\pm\vecv_3,\ldots$ as above we set $\ell_j=|\vecv_j|$ (thus $0<\ell_1\leq \ell_2\leq \ell_3\leq\ldots$), and also define
\begin{align*}
 \mathcal V_j:=\frac{\pi^{n/2}}{\Gamma(\frac{n}{2}+1)}\ell_j^n\,,
\end{align*}
so that $\mathcal V_j$ is the volume of an $n$-dimensional ball of radius $\ell_j$. Our main result in \cite{jag} states that, as $n\to\infty$, the volumes $\{\mathcal V_j\}_{j=1}^{\infty}$ determined by a random lattice $L\in X_n$ behave like the points of a Poisson process on the positive real line with constant intensity $\frac{1}{2}$. 

In the present paper we investigate also the distribution of the angles between $\vecv_1,\ldots,\vecv_N$ for a random lattice $L\in X_n$.  Since the vectors $\{\vecv_j\}_{j=1}^{\infty}$ are determined only up to sign, the angles between them are a priori not well-defined. We avoid this ambiguity by introducing a "symmetrized" angle measure $\varphi$, taking values in the interval $[0,\frac{\pi}{2}]$. To be more specific, we denote the Euclidean angle between the vectors $\vecx_1,\vecx_2\in\R^n\setminus\{\vec0\}$ by $\phi(\vecx_1,\vecx_2)$ and define 
\begin{align*}
\varphi(\vecx_1,\vecx_2):=\begin{cases}
\phi(\vecx_1,\vecx_2)&\text{if $\phi(\vecx_1,\vecx_2)\in[0,\frac{\pi}{2}]$,}\\
\pi-\phi(\vecx_1,\vecx_2)&\text{otherwise.}
\end{cases}
\end{align*}
Given $L\in X_n$ and $i,j\in\Z_{\geq1}$,  we let $\varphi_{ij}:=\varphi(\vecv_i,\vecv_j)$. Our first result states that for a random lattice $L\in X_n$ the angles $\{\varphi_{ij}\}_{i<j}$ accumulate to $\frac{\pi}{2}$, as $n\to\infty$, with a rate comparable with $n^{-\frac{1}{2}}$. 

\begin{prop}\label{concentrate}
For any fixed $N\in\Z_{\geq2}$, the probability 
\begin{align*}
\text{Prob}_{\mu_n}\Big\{L\in X_n\,\,\big|\,\,\exists i<j\leq N:\sfrac{\pi}{2}-\varphi_{ij}>\sfrac{C}{\sqrt n}\Big\}
\end{align*}
tends to $0$ as $C,n\to\infty$.
\end{prop}

Proposition \ref{concentrate} suggests that it is natural to study the normalized variables 
\begin{align*}
\widetilde \varphi_{ij}:=\sqrt{n}\big(\sfrac{\pi}{2}-\varphi_{ij}\big)=\sqrt{n}\big(\sfrac{\pi}{2}-\varphi(\vecv_i,\vecv_j)\big).
\end{align*} 
Given an integer $N\geq2$, we study the joint distribution of the random variables $\mathcal{V}_j$, $1\leq j\leq N$, and $\widetilde \varphi_{ij}$, $1\leq i<j\leq N$. Our main result is the following theorem, where we use the term \emph{positive Gaussian variable} to denote a random variable  $\Phi$ satisfying $\Phi=|X|$ for a random variable $X\in N(0,1)$.

\begin{thm}\label{poissongauss}
Let $N\in\Z_{\geq2}$. 
The joint distribution of $\mathcal{V}_1,\ldots,\mathcal{V}_N$ and $\widetilde\varphi_{ij}$, $1\leq i<j\leq N$, converges, as $n\to\infty$, to the joint distribution of the first $N$ points of a Poisson process on the positive real line with intensity $\frac{1}{2}$ and a collection of $\binom{N}{2}$ independent positive Gaussian variables (which are also independent of the first $N$ variables).
\end{thm}

As was mentioned above the limit distribution of the volumes $\{\mathcal V_j\}_{j=1}^{\infty}$ alone was determined in \cite[Thm.\ 1]{jag}. Let us also point out that the limit distribution of the variables $\widetilde \varphi_{ij}$ is natural, since it is exactly the same distribution as one gets as the asymptotic distribution of the angles between $N$ independent random unit vectors in $\R^n$, chosen from a uniform distribution on $S^{n-1}$, as $n\to\infty$; cf.\ \cite[Thm. 4]{stam}. We give a new proof of this result using our set-up, in Lemma \ref{anglelemma} and Theorem \ref{randomdirthm} below.

It is further possible to reformulate Theorem \ref{poissongauss} in the dual setting of eigenvalues and eigenfunctions of the Laplacian on flat tori $\R^n/L$ with $L\in X_n$. It is well-known that the eigenvalues of the torus $\R^n/L$ are $4\pi^2|\vecell|^2$, with $\vecell$ belonging to the dual lattice $L^*$, and that the corresponding eigenfunctions are $f_{\vecell}(\vecx):=e^{2\pi i\langle \vecell,\vecx\rangle}$. Note that, for $\vecell\neq\vec0$, the functions $f_{\vecell}$ are complex wave functions propagating in the direction of the vector $\vecell$. We also recall that "desymmetrizing" and renormalizing the eigenvalues to have mean spacing $1$ yields the sequence $\{\sfrac{1}{2}\mathcal V_j\}_{j=1}^{\infty}$ for the lattice $L^*$. In this setting Theorem \ref{poissongauss} states:

\begin{thm}
Let $N\in\Z_{\geq2}$. For a random flat torus $\R^n/L$ with $L\in X_n$, the joint distribution of the $N$ first non-zero eigenvalues ("desymmetrized" and normalized to have mean-spacing $1$) and the $\binom{N}{2}$ properly $\sqrt n$-normalized angles between the directions of propagation of the corresponding eigenfunctions converges, as $n\to\infty$, to the joint distribution of the first $N$ 
points of a Poisson process on the positive real line with intensity $1$ and a collection of $\binom{N}{2}$ independent positive Gaussian variables (which are also independent of the first $N$ variables). 
\end{thm}

We also mention that Proposition \ref{concentrate} can be used to determine the limit distribution of any fixed number of \textit{successive minima} of a random lattice $L\in X_n$ as $n\to\infty$. In fact we prove that for each $N\in\Z_{\geq1}$, the $N$-tuple of the first $N$ successive minima, suitably normalized, has the same limit distribution as the $N$-tuple $(\mathcal V_1,\ldots,\mathcal V_N)$ as $n\to\infty$ (cf.\ Corollary \ref{succmin}). 

We end the introduction with an outline of the paper. In Section \ref{accumulate} we prove Proposition \ref{concentrate} and some related results using Rogers' mean value formula \cite{rogers1} and the estimates in \cite[Sec.\ 3]{jag}. In Section \ref{randomdirections} we treat the asymptotic distribution of angles between random directions in $\R^n$. Although the results here are known (cf.\ \cite[Thm. 4]{stam}), we give a detailed presentation of this topic since it gives us the opportunity to introduce some arguments used also in the more involved context of Theorem \ref{poissongauss}. In Sections \ref{expectation} and \ref{proofsec} we discuss the proof of Theorem \ref{poissongauss}. In Section \ref{expectation} we prove, using Rogers' formula, the convergence of the expectation values of certain series of functions, depending on the sequences $\{\mathcal V_j\}_{j=1}^{\infty}$ and $\{\widetilde\varphi_{ij}\}_{i<j}$, as $n\to\infty$. The proof of Theorem \ref{poissongauss} is then concluded in Section \ref{proofsec} with an inclusion-exclusion argument. Finally, in Section  \ref{sucmin} we discuss the application of Proposition \ref{concentrate} to successive minima.

\section{The angles $\varphi_{ij}$ accumulate to $\frac{\pi}{2}$}\label{accumulate}

For $V>0$ and $0\leq\varphi_1<\varphi_2\leq\frac{\pi}{2}$ we consider the function $$f_{V,\varphi_1,\varphi_2}:(\R^n)^2\to\{0,1\}$$ defined by
\begin{align*}
f_{V,\varphi_1,\varphi_2}(\vecx_1,\vecx_2)=
I\big(\vecx_1,\vecx_2\in B_V\setminus\{\vec0\}\:;\:\vecx_1\neq\pm\vecx_2\:;\:\varphi(\vecx_1,\vecx_2)\in[\varphi_1,\varphi_2]\big),
\end{align*}
where $I(\cdot)$ is the indicator function and $B_V\subset\R^n$ is the closed $n$-ball of volume $V$ centered at the origin. We also set
\begin{align*}
M_{V,\varphi_1,\varphi_2}(L):=\frac{1}{8}\sum_{\vecm_1,\vecm_2\in L\setminus\{\vec0\}}f_{V,\varphi_1,\varphi_2}(\vecm_1,\vecm_2).
\end{align*}
Recall from the introduction that for any given lattice $L\in X_n$, we choose $\vecv_1,\vecv_2,\ldots\in L$ so that $0<|\vecv_1|\leq|\vecv_2|\leq\ldots$, $L=\{\vec0,\pm\vecv_1,\pm\vecv_2,\ldots\}$ and $\vecv_j\neq\pm\vecv_k$ for $j\neq k$. Thus, for given $L$, the vectors $\vecv_j$ are uniquely determined up to sign and permutation of vectors of equal length. It follows that $M_{V,\varphi_1,\varphi_2}$ is the random variable on $X_n$ which counts the number of unordered pairs of distinct non-zero lattice vectors $\vecm_1,\vecm_2\in \{\vecv_j\}_{j=1}^{\infty}\cap B_V$ with $\varphi(\vecm_1,\vecm_2)\in[\varphi_1,\varphi_2]$. 

We get our first results by studying the expectation value of $M_{V,\varphi_1,\varphi_2}$. It follows immediately from Rogers' mean value formula (cf.\ \cite{rogers1}) and the estimates in \cite[Sec.\ 3]{jag} that, for $n\geq3$,
\begin{align}\label{huvud1}
\mathbb E\big(M_{V,\varphi_1,\varphi_2}(\cdot)\big)&=\frac{1}{8}\bigg(\int_{\R^n}\int_{\R^n}f_{V,\varphi_1,\varphi_2}(\vecx_1,\vecx_2)\,d\vecx_1d\vecx_2\\
&+\int_{\R^n}f_{V,\varphi_1,\varphi_2}(\vecx,\vecx)\,d\vecx+\int_{\R^n}f_{V,\varphi_1,\varphi_2}(\vecx,-\vecx)\,d\vecx\bigg)+R(n),\nonumber
\end{align}
where $0\leq R(n)\ll2^{-n}$. Here the implied constant depends on $V$ but not on $\varphi_1$ or $\varphi_2$. From the definition of $f_{V,\varphi_1,\varphi_2}$ we get that the two last integrals in \eqref{huvud1} equal zero. Hence 
\begin{align}\label{huvud2}
\mathbb E\big(M_{V,\varphi_1,\varphi_2}(\cdot)\big)&=\frac{1}{8}\int_{\R^n}\int_{\R^n}f_{V,\varphi_1,\varphi_2}(\vecx_1,\vecx_2)\,d\vecx_1d\vecx_2+O(2^{-n}).
\end{align}

\begin{lem}\label{Lemma1}
Let $V>0$ and $0\leq\varphi_1<\varphi_2<\frac{\pi}{2}$ be fixed. Then 
\begin{align*}
\lim_{n\to\infty}\mathbb E\big(M_{V,\varphi_1,\varphi_2}(\cdot)\big)=0. 
\end{align*}
\end{lem}

\begin{proof}
By writing the integral in \eqref{huvud2} as an iterated integral and changing to spherical coordinates in the inner integral we find that
\begin{multline}\label{sinint}
\int_{\R^n}\int_{\R^n}f_{V,\varphi_1,\varphi_2}(\vecx_1,\vecx_2)\,d\vecx_1d\vecx_2\\
=2V\frac{\omega_{n-1}{R_V}^n}{n}\int_{\varphi_1}^{\varphi_2}\sin^{n-2}(\phi)\,d\phi=2V^2\frac{\omega_{n-1}}{\omega_n}\int_{\varphi_1}^{\varphi_2}\sin^{n-2}(\phi)\,d\phi, 
\end{multline}
where $\omega_n$ is the $(n-1)$-dimensional volume of the unit sphere $S^{n-1}\subset\R^n$ and $R_V$ is the radius of the ball $B_V$. Recalling that
\begin{align}\label{omega}
\omega_n=\frac{2\pi^{n/2}}{\Gamma(n/2)} 
\end{align}
and using Stirling's formula we conclude that 
\begin{align*}
\int_{\R^n}\int_{\R^n}f_{V,\varphi_1,\varphi_2}(\vecx_1,\vecx_2)\,d\vecx_1d\vecx_2\ll V^2\sqrt{n}\sin^{n-2}(\varphi_2), 
\end{align*}
which clearly implies the desired result. 
\end{proof}

Lemma \ref{Lemma1} implies that the angles $\varphi_{ij}$ ($1\leq i<j\leq N$) accumulate to $\frac{\pi}{2}$ (cf.\ the proof of Proposition \ref{concentrate} below). In order to determine the rate of accumulation we also need the following lemma. The statement involves the error function, which is defined by
\begin{align*}
\mathrm{erf}(x):=\frac{2}{\sqrt{\pi}}\int_0^xe^{-t^2}\,dt.
\end{align*}

\begin{lem}\label{Cint}
Let $C>0$ be fixed. Then 
\begin{align*}
\frac{\omega_{n-1}}{\omega_n}\int_{\frac{\pi}{2}-\frac{C}{\sqrt{n}}}^{\frac{\pi}{2}}\sin^{n-2}(\phi)\,d\phi\to\frac{1}{2}\mathrm{erf}\Big(\frac{C}{\sqrt{2}}\Big)
\end{align*}
as $n\to\infty$.
\end{lem}

\begin{proof}
The change of variables $\phi=\frac{\pi}{2}-\frac{t}{\sqrt{n}}$ yields
\begin{align*}
\frac{\omega_{n-1}}{\omega_n}\int_{\frac{\pi}{2}-\frac{C}{\sqrt{n}}}^{\frac{\pi}{2}}\sin^{n-2}(\phi)\,d\phi=\frac{\omega_{n-1}}{\omega_n\sqrt{n}}\int_{0}^{C}\cos^{n-2}\Big(\frac{t}{\sqrt{n}}\Big)\,dt. 
\end{align*}
It follows from \eqref{omega} and a slightly more careful application of Stirling's formula that 
\begin{align}\label{volumequotient}
\lim_{n\to\infty}\frac{\omega_{n-1}}{\omega_n\sqrt{n}}=\frac{1}{\sqrt{2\pi}}.
\end{align}
By taking the logarithm and then using Taylor expansions we further obtain the pointwise limit 
\begin{align}\label{coslimit}
\lim_{n\to\infty}\cos^{n-2}\Big(\frac{t}{\sqrt{n}}\Big)=e^{-\frac{t^2}{2}}. 
\end{align}
Hence, the lemma follows by the dominated convergence theorem. 
\end{proof}

\begin{prop}\label{concentration}
Let $V>0$. For every $\ve>0$ there exist $C>0$ and $n_0\in\Z_{\geq1}$ such that
\begin{align*}
\text{Prob}_{\mu_n}\Big\{L\in X_n\,\,\big|\,\,\exists\vecm_1,\vecm_2\in (L\cap B_V)\setminus\{\vec0\}:\vecm_1\neq\pm\vecm_2\,,\,\sfrac{\pi}{2}-\varphi(\vecm_1,\vecm_2)>\sfrac{C}{\sqrt n}\Big\}<\ve
\end{align*}
for all $n\geq n_0$.
\end{prop}

\begin{proof}
For given $V>0$ and $C>0$ we have, by \eqref{huvud2}, \eqref{sinint} and Lemma \ref{Cint},
\begin{align}\label{huvud3}
\mathbb E\big(M_{V,\frac{\pi}{2}-\frac{C}{\sqrt{n}},\frac{\pi}{2}}(\cdot)\big)&\to\frac{V^2}{8}\mathrm{erf}\Big(\frac{C}{\sqrt{2}}\Big)
\end{align}
as $n\to\infty$. Since
\begin{align*}
 \lim_{C\to\infty}\frac{V^2}{8}\mathrm{erf}\Big(\frac{C}{\sqrt{2}}\Big)=\frac{V^2}{8}
\end{align*}
and this limit coincides with the value of $\lim_{n\to\infty}\mathbb E\big(M_{V,0,\frac{\pi}{2}}(\cdot)\big)$, the proposition follows.
\end{proof}

\begin{proof}[Proof of Proposition \ref{concentrate}]
We know from \cite[Thm.\ 3]{rogers3} (or \cite[Thm.\ 1]{jag}) that 
\begin{align}\label{rogersthm3}
\text{Prob}_{\mu_n}\big\{L\in X_n\,\,\big|\,\,\mathcal V_N\leq V\big\}&=\text{Prob}_{\mu_n}\big\{L\in X_n\,\,\big|\,\,\#\{j:\mathcal{V}_j\leq V\}\geq N\big\}\nonumber\\
&\to1-e^{-V/2}\sum_{k=0}^{N-1}\Big(\frac{V}{2}\Big)^k\frac{1}{k!}
\end{align}
as $n\to\infty$. By choosing $V$ large enough we can make the right hand side of \eqref{rogersthm3} as close to $1$ as we like. The result follows from this and Proposition \ref{concentration}.
\end{proof}

\section{The distribution of angles between random directions in $\R^n$}\label{randomdirections}

Let $N\geq2$. In this section we discuss the asymptotic distribution of the angles between $N$ independent random unit vectors $\vecu_1,\ldots,\vecu_N$ in $\R^n$, chosen from a uniform distribution on $S^{n-1}$, as $n\to\infty$. For $1\leq i<j\leq N$ let $\alpha_{ij}=\phi(\vecu_i,\vecu_j)$ denote the angle between $\vecu_i$ and $\vecu_j$, and set $\widetilde\alpha_{ij}=\sqrt{n}(\alpha_{ij}-\frac{\pi}{2})$. The following well-known result seems to first have been noted (in an equivalent form) by Borel \cite[Chap. V]{borel}; cf.\ \cite[Sec.\ 6.1]{DF}.

\begin{lem}\label{anglelemma}
Assume that $N=2$. Then $\widetilde\alpha_{12}$ converges in distribution to $N(0,1)$ as $n\to\infty$.
\end{lem}

\begin{proof}
For all $c<c'$ we have
\begin{multline*}
\text{Prob}_{n}\Big\{
c<\widetilde\alpha_{12}\leq c'\Big\}=\text{Prob}_{n}\Big\{ \frac{\pi}{2}+\frac{c}{\sqrt{n}}<\alpha_{12}\leq \frac{\pi}{2}+\frac{c'}{\sqrt{n}}\Big\}\\
=\frac{1}{\omega_n^2}\int_{S^{n-1}}\int_{S^{n-1}}I\Big(\frac{\pi}{2}+\frac{c}{\sqrt{n}}<\alpha_{12}\leq\frac{\pi}{2}+\frac{c'}{\sqrt{n}}\Big)\,d\sigma(\vecu_1)d\sigma(\vecu_2),
\end{multline*}
where $I(\cdot)$ is the indicator function, $d\sigma$ denotes the $(n-1)$-dimensional volume measure on $S^{n-1}$ and $\omega_n$ is the volume of $S^{n-1}$. By first changing to spherical coordinates and then letting $\phi=\frac{\pi}{2}+\frac{t}{\sqrt n}$ we obtain, for all sufficiently large $n$ (depending on $c,c'$), 
\begin{align*}
&\text{Prob}_{n}\Big\{c<\widetilde\alpha_{12}\leq c'\Big\}=\frac{\omega_{n-1}}{\omega_n}\int_{\frac{\pi}{2}+\frac{c}{\sqrt{n}}}^{\frac{\pi}{2}+\frac{c'}{\sqrt{n}}}\sin^{n-2}(\phi)\,d\phi=\frac{\omega_{n-1}}{\omega_n\sqrt{n}}\int_{c}^{c'}\cos^{n-2}\Big(\frac{t}{\sqrt{n}}\Big)\,dt. 
\end{align*}
Finally, by \eqref{volumequotient}, \eqref{coslimit} and the dominated convergence theorem, we find that
\begin{align*}
\text{Prob}_{n}\Big\{c<\widetilde\alpha_{12}\leq c'\Big\}\to\frac{1}{\sqrt{2\pi}}\int_{c}^{c'}e^{-\frac{t^2}{2}}\,dt\end{align*}
as $n\to\infty$, which completes the proof of the lemma.
\end{proof}

We continue by considering the case $N\geq3$. It is clear that for a fixed dimension $n$ the normalized angles  $\widetilde\alpha_{ij}$, $1\leq i<j\leq N$, are dependent as random variables on $(S^{n-1})^N$. Nevertheless, as $n\to\infty$, these variables converge in distribution to a collection of independent normally distributed variables. In precise terms:

\begin{thm}\label{randomdirthm}
Let $N\geq3$ and $c_{ij}<c_{ij}'$ for $1\leq i<j\leq N$. Then
\begin{align*}
&\text{Prob}_{n}\Big\{c_{ij}<\widetilde\alpha_{ij}\leq c_{ij}':1\leq i<j\leq N\Big\}\to\prod_{1\leq i<j\leq N}\frac{1}{\sqrt{2\pi}}\int_{c_{ij}}^{c_{ij}'}e^{-\frac{t^2}{2}}\,dt,
\end{align*}
as $n\to\infty$. In other words the joint distribution of the variables $\widetilde\alpha_{ij}$ converges to the joint distribution of $\binom{N}{2}$ independent Gaussian variables as $n\to\infty$. 
\end{thm}

This theorem follows from \cite[Thm. 4]{stam}, where in fact it is proved that we even have convergence in total variation. We give a detailed proof of Theorem \ref{randomdirthm} here using our set-up, since it gives us the opportunity to introduce some arguments used also in the more involved context of Theorem \ref{poissongauss}.

\begin{proof}[Proof of Theorem \ref{randomdirthm}]
Fix $c_{ij}<c_{ij}'$ for $1\leq i<j\leq N$. We have
\begin{align}\label{probability}
&\text{Prob}_{n}\Big\{c_{ij}<\widetilde\alpha_{ij}\leq c_{ij}':1\leq i<j\leq N\Big\}\\
&=\frac{1}{\omega_n^N}\int_{S^{n-1}}\cdots\int_{S^{n-1}}I\Big(c_{ij}<\widetilde\alpha_{ij}\leq c_{ij}':1\leq i<j\leq N\Big)\,d\sigma(\vecu_1)\ldots d\sigma(\vecu_N).\nonumber
\end{align}
Due to rotational symmetry the probability in \eqref{probability} equals (when $n\geq N$)
\begin{multline*}
\bigg(\prod_{\ell=1}^{N-1}\frac{\omega_{n-\ell}}{\omega_n}\bigg)\int_{\phi_{12}=0}^{\pi}\cdots\int_{\phi_{(N-1)N}=0}^{\pi}I\Big(c_{ij}<\widetilde\alpha_{ij}\leq c_{ij}':1\leq i<j\leq N\Big)\\
\times\prod_{1\leq i<j\leq N}\sin^{n-i-1}(\phi_{ij})\,d\phi_{(N-1)N}\ldots d\phi_{13}d\phi_{12}\,,
\end{multline*}
where we only consider unit vectors on the form
\begin{align}\label{uvectors}
\vecu_1&=(1,0,\ldots,0)\nonumber\\
\vecu_2&=(\cos \phi_{12},\sin \phi_{12},0,\ldots,0)\nonumber\\
\vecu_3&=(\cos \phi_{13},\sin \phi_{13}\cos \phi_{23},\sin \phi_{13}\sin \phi_{23},0,\ldots,0)\\
&\,\,\,\vdots\nonumber\\
\vecu_N&=(\cos \phi_{1N},\sin \phi_{1N}\cos \phi_{2N},\ldots,\sin \phi_{1N}\cdots\sin \phi_{(N-1)N},0,\ldots,0).\nonumber
\end{align}  
By the change of variables $\phi_{ij}=\frac{\pi}{2}+\frac{t_{ij}}{\sqrt n}$, $i\leq i<j\leq N$, we get
\begin{align}\label{serious}
&\text{Prob}_{n}\Big\{c_{ij}<\widetilde\alpha_{ij}\leq c_{ij}':1\leq i<j\leq N\Big\}=(n^{-\frac{1}{2}})^{\binom{N}{2}}\bigg(\prod_{\ell=1}^{N-1}\prod_{m=1}^{\ell}\frac{\omega_{n-m}}{\omega_{n-m+1}}\bigg)\\
&\times\int_{t_{12}=-\sqrt n\frac{\pi}{2}}^{\sqrt n\frac{\pi}{2}}\cdots\int_{t_{(N-1)N}=-\sqrt n\frac{\pi}{2}}^{\sqrt n\frac{\pi}{2}}I\Big(c_{ij}<\widetilde\alpha_{ij}\leq c_{ij}':1\leq i<j\leq N\Big)\nonumber\\
&\times\prod_{1\leq i<j\leq N}\cos^{n-i-1}\Big(\frac{t_{ij}}{\sqrt n}\Big)\,dt_{(N-1)N}\ldots dt_{13}dt_{12}.\nonumber
\end{align} 
We note that the double product in the first line of \eqref{serious} has $\binom{N}{2}$ factors. Hence, by \eqref{volumequotient}, we obtain
\begin{align}\label{superquotient}
\lim_{n\to\infty}(n^{-\frac{1}{2}})^{\binom{N}{2}}\prod_{\ell=1}^{N-1}\prod_{m=1}^{\ell}\frac{\omega_{n-m}}{\omega_{n-m+1}}=(2\pi)^{-\frac{1}{2}\binom{N}{2}}.
\end{align}
Also, by the same argument as in \eqref{coslimit}, we have the pointwise limit
\begin{align}\label{hm}
\lim_{n\to\infty}\prod_{1\leq i<j\leq N}\cos^{n-i-1}\Big(\frac{t_{ij}}{\sqrt n}\Big)=\prod_{1\leq i<j\leq N}e^{-\frac{{t_{ij}}^2}{2}}.
\end{align}

It remains to understand for what values of the variables $t_{ij}$ ($1\leq i<j\leq N$) the indicator function in \eqref{serious} is non-zero as $n\to\infty$. Since $|\vecu_1|=\cdots=|\vecu_N|=1$ we have
\begin{align}\label{innerproduct}
\vecu_i\cdot\vecu_j=\cos \alpha_{ij},\qquad 1\leq i<j\leq N.
\end{align}
Let us first consider the case $i=1$. By the form of the vectors in \eqref{uvectors} it is immediate that
\begin{align}\label{i=1}
\vecu_1\cdot\vecu_j=\cos \phi_{1j},\qquad 2\leq j\leq N.
\end{align}
Hence $\phi_{1j}=\alpha_{1j}$, and thus $t_{1j}=\widetilde\alpha_{1j}$, for $2\leq j\leq N$. We conclude that a necessary condition for the integrand in  \eqref{serious} to be non-zero is that $c_{1j}<t_{1j}\leq c_{1j}'$ for $2\leq j\leq N$.

In the general case $1\leq i<j\leq N$, it follows  from \eqref{uvectors} that
\begin{align}\label{i<j}
\vecu_i\cdot\vecu_j&=\sum_{m=1}^{i-1}\sin \Big(\frac{t_{mi}}{\sqrt n}\Big)\sin \Big(\frac{t_{mj}}{\sqrt n}\Big)\prod_{k=1}^{m-1}\cos \Big(\frac{t_{ki}}{\sqrt n}\Big)\cos \Big(\frac{t_{kj}}{\sqrt n}\Big)\\
&-\sin \Big(\frac{t_{ij}}{\sqrt n}\Big)\prod_{k=1}^{i-1}\cos \Big(\frac{t_{ki}}{\sqrt n}\Big)\cos \Big(\frac{t_{kj}}{\sqrt n}\Big).\nonumber
\end{align}
Set $K=1+\max\{|c_{ij}|,|c_{ij}'|:1\leq i<j\leq N\}$. Now if $(t_{12},\ldots,t_{(N-1)N})$ lies outside the box $[-K,K]^{\binom{N}{2}}$, then there are some indices $1\leq i<j\leq N$ such that $|t_{ij}|>K$ but $|t_{i'j'}|\leq K$ for all $1\leq i'<j'<j$ and $|t_{i'j}|\leq K$ for all $1\leq i'<i$. Then, by Taylor expansions in \eqref{i<j}, we get $\vecu_i\cdot\vecu_j=-\sin(t_{ij}/\sqrt n)+O(n^{-1})$, and thus 
\begin{align*}
\widetilde\alpha_{ij}=\sqrt n\arcsin\Big(\sin\Big(\frac{t_{ij}}{\sqrt n}\Big)+O(n^{-1})\Big),
\end{align*}
where the implied constant depends only on $K$ and $N$. For $n$ sufficiently large (as only depends on $K,N$) this implies $|\widetilde\alpha_{ij}|>K-1\geq\max(|c_{ij}|,|c_{ij}'|)$, so that the integrand in \eqref{serious} vanishes. Hence we have proved that for $n$ large, we may restrict the domain of integration in \eqref{serious} to $[-K,K]^{\binom{N}{2}}$. 

But for any $(t_{12},\ldots,t_{(N-1)N})\in[-K,K]^{\binom{N}{2}}$, it follows from \eqref{innerproduct} and Taylor expansions in \eqref{i<j} that $\widetilde\alpha_{ij}=t_{ij}+O(n^{-1/2})$, where the implied constant depends only on $K$ and $N$. Hence, for any fixed $\ve>0$ and large enough $n$, we get an upper (lower) bound for \eqref{serious} by replacing "$c_{ij}<\widetilde\alpha_{ij}\leq c_{ij}'$" by "$c_{ij}-\ve<t_{ij}<c_{ij}'+\ve$" ("$c_{ij}+\ve<t_{ij}<c_{ij}'-\ve$"). Using also \eqref{superquotient}, \eqref{hm} and the dominated convergence theorem, we conclude:
\begin{align*}
\limsup_{n\to\infty}\text{Prob}_n\Big\{c_{ij}<\widetilde\alpha_{ij}\leq c_{ij}':1\leq i<j\leq N\Big\}\leq\prod_{1\leq i<j\leq N}\frac{1}{\sqrt{2\pi}}\int_{c_{ij}-\ve}^{c_{ij}'+\ve}e^{-\frac{t^2}{2}}\,dt,
\end{align*}
and (assuming $\ve<\frac{1}{2}(c_{ij}'-c_{ij})$ for all $i,j$)
\begin{align*}
\liminf_{n\to\infty}\text{Prob}_n\Big\{c_{ij}<\widetilde\alpha_{ij}\leq c_{ij}':1\leq i<j\leq N\Big\}\geq\prod_{1\leq i<j\leq N}\frac{1}{\sqrt{2\pi}}\int_{c_{ij}+\ve}^{c_{ij}'-\ve}e^{-\frac{t^2}{2}}\,dt.
\end{align*}
The proof is now concluded by letting $\ve\to0$.
\end{proof}

\section{Convergence of expectation values}\label{expectation}

We begin this section by fixing some notation concerning the limiting variables in Theorem \ref{poissongauss}. We introduce the Poisson process  $\{N(t),t\geq0\}$, defined on the positive real line with constant intensity $\frac{1}{2}$, and let $T_1, T_2,T_3,\ldots$ denote the points of the process ordered in such a way that $0<T_1<T_2< T_3<\ldots$. We recall that $N(t)$ denotes the number of points falling in the interval $(0,t]$ and that $N(t)$ is Poisson distributed with expectation value $\frac{1}{2}t$. Furthermore, we let $\{\Phi_{ij}\}_{1\leq i<j<\infty}$ be a family of independent positive Gaussian variables.

Our first approach to trying to prove Theorem \ref{poissongauss} was to study moments of counting variables of the type discussed in Section \ref{accumulate} (see also \cite{jag}). In particular, for $k\geq2$, $0<V_1\leq V_2\leq\ldots\leq V_k$ and $c_{ij}\in\R_{\geq0}$, $1\leq i<j\leq k$, introduce the function $$f_{\{V_j\}_{j=1}^k,\{c_{ij}\}_{1\leq i<j\leq k}}:(\R^n)^k\to\{0,1\},$$ defined by
\begin{align*}
f_{\{V_j\}_{j=1}^k,\{c_{ij}\}_{1\leq i<j\leq k}}(\vecx_1,\ldots,\vecx_k)=I\big(\vecx_j\in B_{V_j}\setminus\{\vec0\}\:;\:
\varphi(\vecx_i,\vecx_j)\in\big[\sfrac{\pi}{2}-\sfrac{c_{ij}}{\sqrt n},\sfrac{\pi}{2}\big]\big),
\end{align*}
and the related random variable on $X_n$ defined by
\begin{align*}
M_{\{V_j\}_{j=1}^k,\{c_{ij}\}_{1\leq i<j\leq k}}(L):=\frac{1}{2^{k}}\sum_{\vecm_1,\ldots,\vecm_k\in L\setminus\{\vec0\}}f_{\{V_j\}_{j=1}^k,\{c_{ij}\}_{1\leq i<j\leq k}}(\vecm_1,\ldots,\vecm_k).
\end{align*}
Then $M_{\{V_j\}_{j=1}^k,\{c_{ij}\}_{1\leq i<j\leq k}}(L)$ equals the number of $k$-tuples of non-zero pairs of lattice vectors $\pm\vecm_1,\ldots,\pm\vecm_k$ satisfying $\pm \vecm_j\in L\cap B_{V_j}$ and $\varphi(\vecm_i,\vecm_j)\in\big[\sfrac{\pi}{2}-\sfrac{c_{ij}}{\sqrt n},\sfrac{\pi}{2}\big]$. Using a mixture of the methods appearing in this section and \cite{jag} it is possible to calculate the limits, as $n\to\infty$, of all moments of  $M_{\{V_j\}_{j=1}^k,\{c_{ij}\}_{1\leq i<j\leq k}}(L)$. Furthermore, these limits can be shown to coincide with the moments of the corresponding counting variable for the limit distribution described in Theorem \ref{poissongauss}. However, for $k\geq 3$ it is not clear whether this limiting counting variable is determined by its moments or not. Thus it is also not clear how to find a proof of Theorem \ref{poissongauss} in this direction.

Our proof of Theorem \ref{poissongauss}, presented in this section and the next, instead follows an approach suggested to us by Svante Janson. For each $j\in\Z_{\geq2}$ we let $M_j$ denote the set of all $j$-tuples $(n_1,\ldots,n_j)\in\big(\Z_{\geq1}\big)^{j}$ with pairwise distinct entries. 

\begin{thm}\label{rogersthm}
Let $k\in\Z_{\geq2}$ and  $\ell\in\Z_{\geq0}$. Then, for all bounded Borel measurable functions $f:(\R_{\geq0})^{k+\ell+\binom{k}{2}}\to\R$ with compact support,
\begin{multline}\label{expect}
\mathbb E\Big(\sum_{(n_1,\ldots,n_{k+\ell})\in M_{k+\ell}} f\big(\mathcal{V}_{n_1},\ldots,\mathcal{V}_{n_{k+\ell}},\widetilde\varphi_{n_1n_2},\widetilde\varphi_{n_1n_3},\ldots,\widetilde\varphi_{n_{k-1}n_k}\big)\Big)\\
\to\mathbb E\Big(\sum_{(n_1,\ldots,n_{k+\ell})\in M_{k+\ell}} f\big(T_{n_1},\ldots,T_{n_{k+\ell}},\Phi_{n_1n_2},\Phi_{n_1n_3},\ldots,\Phi_{n_{k-1}n_k}\big)\Big) 
\end{multline}
as $n\to\infty$. 
\end{thm}

\begin{proof}
 Set $\lambda=k+\ell$ and fix a bounded Borel measurable function $f:(\R_{\geq0})^{\lambda+\binom{k}{2}}\to\R$ with compact support. Given $n$ we define the related function $\widetilde f:(\R^n)^{\lambda}\to\R$ by
 \begin{align*}
 &\widetilde f(\vecx_1,\ldots,\vecx_{\lambda})\\
 &=\begin{cases}
 f\big(V_n|\vecx_1|^n,\ldots,V_n|\vecx_{\lambda}|^n,\widetilde{\varphi}(\vecx_1,\vecx_2),\ldots,\widetilde{\varphi}(\vecx_{k-1},\vecx_k)\big)& \text{ if $\vecx_i\neq\vec0$ for $1\leq i\leq k$}\\
 0& \text{ otherwise,}
 \end{cases}
\end{align*}
where  $V_n$ is the volume of the $n$-dimensional unit ball and $\widetilde\varphi$ is given by 
\begin{align*}
\widetilde\varphi(\vecu,\vecv)=\sqrt n\big(\sfrac{\pi}{2}-\varphi(\vecu,\vecv)\big).
\end{align*}
Using Rogers' mean value formula (cf.\ \cite[Thm. 4]{rogers1} and \cite[Sec.\ 2]{jag} for details on the notation and terminology) we find that, for each sufficiently large $n$,
\begin{align}\label{rogers}
&\mathbb E\Big(\sum_{(n_1,\ldots,n_{\lambda})\in M_{\lambda}} f\big(\mathcal{V}_{n_1},\ldots,\mathcal{V}_{n_{\lambda}},\widetilde\varphi_{n_1n_2},\ldots,\widetilde\varphi_{n_{k-1}n_k}\big)\Big)\nonumber\\
&=\frac{1}{2^{\lambda}}\int_{X_n}\sum_{\vecm_1,\ldots,\vecm_{\lambda}\in L\setminus\{\vec0\}}\widetilde f(\vecm_1,\ldots,\vecm_{\lambda})I(\vecm_i=\pm\vecm_j\Leftrightarrow i= j)\,d\mu_n(L)\nonumber\\
&=\frac{1}{2^{\lambda}}\int_{\R^n}\cdots\int_{\R^n}\widetilde f(\vecx_1,\ldots,\vecx_{\lambda})I(\vecx_i=\pm\vecx_j\Leftrightarrow i= j)
\,d\vecx_1\ldots d\vecx_{\lambda}\\
&+\frac{1}{2^{\lambda}}\sum_{(\nu,\mu)}\sum_{q=1}^{\infty}\sum_{D}\Big(\frac{e_1}{q}\cdots\frac{e_m}{q}\Big)^n\int_{\R^n}\cdots\int_{\R^n}\widetilde f\Big(\sum_{h=1}^m\frac{d_{h1}}{q}\vecx_{h},\ldots,\sum_{h=1}^m\frac{d_{h \lambda}}{q}\vecx_{h}\Big)\nonumber\\
&\times I\Big(\sum_{h=1}^m\frac{d_{h i}}{q}\vecx_{h}=\pm\sum_{h=1}^m\frac{d_{h j}}{q}\vecx_{h}\Leftrightarrow i= j\Big)\,d\vecx_1\ldots d\vecx_m.\nonumber
\end{align}

We will now use the results in \cite[Sec.\ 3]{jag} to estimate the size of the sum in the last two lines  of \eqref{rogers}. First we note that all terms coming from $(\nu,\mu)$-admissible matrices $D$ having entries $d_{ij}\in\{0,\pm1\}$ and exactly one non-zero entry in each column equal zero. This is a consequence of the fact that for such matrices there are repetitions (up to sign) among the arguments in the integrand and thus the indicator function forces the corresponding integrals to be zero.  

Let $K,M>0$ be such that $\text{supp} f\subset[0,K]^{\lambda+\binom{k}{2}}$ and $|f|\leq M$. Then $\text{supp} \widetilde f\subset(B_K)^{\lambda}$ and $|\widetilde f|\leq M$ (recall that $B_K$ is the closed $n$-ball of volume $K$ centered at the origin). Let further $\rho_K$ be the characteristic function of $B_K$. Then, if the sum over $D$ is restricted to all $(\nu,\mu)$-admissible matrices not treated in the previous paragraph, we have
\begin{align*}
&\bigg|\sum_{(\nu,\mu)}\sum_{q=1}^{\infty}\sum_{D}\Big(\frac{e_1}{q}\cdots\frac{e_m}{q}\Big)^n\int_{\R^n}\cdots\int_{\R^n}\widetilde f\Big(\sum_{h=1}^m\frac{d_{h1}}{q}\vecx_{h},\ldots,\sum_{h=1}^m\frac{d_{h \lambda}}{q}\vecx_{h}\Big)\nonumber\\
&\times I\Big(\sum_{h=1}^m\frac{d_{h i}}{q}\vecx_{h}=\pm\sum_{h=1}^m\frac{d_{h j}}{q}\vecx_{h}\Leftrightarrow i= j\Big)\,d\vecx_1\ldots d\vecx_m\bigg|\\
&\leq\sum_{(\nu,\mu)}\sum_{q=1}^{\infty}\sum_{D}\Big(\frac{e_1}{q}\cdots\frac{e_m}{q}\Big)^n\int_{\R^n}\cdots\int_{\R^n}\Big|\widetilde f\Big(\sum_{h=1}^m\frac{d_{h1}}{q}\vecx_{h},\ldots,\sum_{h=1}^m\frac{d_{h \lambda}}{q}\vecx_{h}\Big)\Big|\,d\vecx_1\ldots d\vecx_m\\
&\leq M\sum_{(\nu,\mu)}\sum_{q=1}^{\infty}\sum_{D}\Big(\frac{e_1}{q}\cdots\frac{e_m}{q}\Big)^n\int_{\R^n}\cdots\int_{\R^n}\prod_{j=1}^{\lambda}\rho_K\Big(\sum_{h=1}^m\frac{d_{h j}}{q}\vecx_{h}\Big)\,d\vecx_1\ldots d\vecx_m\ll\Big(\frac{3}{4}\Big)^{\frac{n}{2}},
\end{align*}
where the implied constant depends on $\lambda$, $M$ and $K$ but not on $n$. Here the last step follows from \cite[Sec.\ 9]{rogers2} and \cite[Sec.\ 4]{rogers3} (or \cite[Prop.\ 2 and Lem.\ 2]{jag} with $k=\lambda$ and $V_1=\cdots=V_{\lambda}=K$). As a consequence the expectation value in \eqref{rogers} equals
\begin{align}\label{important}
\frac{1}{2^{\lambda}}\int_{\R^n}\cdots\int_{\R^n}\widetilde f(\vecx_1,\ldots,\vecx_{\lambda})\,d\vecx_1\ldots d\vecx_{\lambda}+O\bigg(\Big(\frac{3}{4}\Big)^{\frac{n}{2}}\bigg).
\end{align}

Next, changing to spherical coordinates and using the rotational symmetry, we find that
\begin{align}\label{spherevariables}
&\int_{\R^n}\cdots\int_{\R^n}\widetilde f(\vecx_1,\ldots,\vecx_{\lambda})\,d\vecx_1\ldots d\vecx_{\lambda}
=\omega_n^{\lambda-k}\bigg(\prod_{h=0}^{k-1}\omega_{n-h}\bigg)\\
&\times\int_{r_1=0}^{\infty}\cdots\int_{r_{\lambda}=0}^{\infty}\int_{\phi_{12}=0}^{\pi}\cdots\int_{\phi_{(k-1)k}=0}^{\pi}f\big(V_nr_1^n,\ldots,V_nr_{\lambda}^n,\widetilde\varphi(\vecu_1,\vecu_2),\ldots,\widetilde\varphi(\vecu_{k-1},\vecu_k)\big)\nonumber\\
&\times\bigg(\prod_{j=1}^{\lambda}r_j^{n-1}\bigg)\prod_{1\leq i<j\leq k}\sin^{n-i-1}(\phi_{ij})\,d\phi_{(k-1)k}\ldots d\phi_{12}dr_{\lambda}\ldots dr_1,\nonumber
\end{align}
where the vectors $\vecu_1,\ldots,\vecu_k$ are given by \eqref{uvectors} with $N=k$. We will here take the angles $$\alpha_{ij}=\phi(\vecu_i,\vecu_j)=\arccos(\vecu_i\cdot\vecu_j)$$ as new variables of integration. It is clear from \eqref{uvectors} that if $(\phi_{12},\ldots,\phi_{(k-1)k})\in(0,\pi)^{\binom{k}{2}}$ then $\vecu_1,\ldots,\vecu_k$ are linearly independent unit vectors and in particular $-1<\vecu_i\cdot\vecu_j<1$ for all $1\leq i<j\leq k$; hence $(\phi_{12},\ldots,\phi_{(k-1)k})\mapsto(\alpha_{12},\ldots,\alpha_{(k-1)k})$ is a $C^{\infty}$ map from $(0,\pi)^{\binom{k}{2}}$ into $(0,\pi)^{\binom{k}{2}}$. This map, which we denote by $J$, is easily seen to be injective (indeed, $\phi_{12}$ is uniquely determined from $\alpha_{12}$; next $\phi_{13},\phi_{23}$ are uniquely determined from $\phi_{12},\alpha_{13},\alpha_{23}$; next $\phi_{14},\phi_{24},\phi_{34}$ are uniquely determined from $\phi_{12},\phi_{13},\phi_{23},\alpha_{14},\alpha_{24},\alpha_{34}$, and so on). Note also that $\frac{\partial\alpha_{ij}}{\partial\phi_{i'j'}}=0$ when $j<j'$ , and when $j=j'$ and $i<i'$; thus when properly ordered the Jacobian matrix of $J$ is lower triangular, and we obtain for its determinant:
\begin{multline*}
\det\frac{\partial(\alpha_{12},\ldots,\alpha_{(k-1)k})}{\partial(\phi_{12},\ldots,\phi_{(k-1)k})}=\prod_{1\leq i<j\leq k}\frac{\partial\alpha_{ij}}{\partial\phi_{ij}}\\
=\prod_{1\leq i<j\leq k}\frac{(\sin \phi_{ij})\big(\prod_{\ell=1}^{i-1}\sin \phi_{\ell i}\sin \phi_{\ell j}\big)}{\sin \alpha_{ij}}=\prod_{1\leq i<j\leq k}\frac{(\sin \phi_{ij})^{k-i}}{\sin \alpha_{ij}}.
\end{multline*} 
This is non-zero for all $(\phi_{12},\ldots,\phi_{(k-1)k})\in(0,\pi)^{\binom{k}{2}}$. Hence $J$ is in fact a $C^{\infty}$ diffeomorphism from $(0,\pi)^{\binom{k}{2}}$ onto an open subset $\Omega\subset(0,\pi)^{\binom{k}{2}}$, and, letting also $s_j=V_nr_j^n$, $1\leq j\leq \lambda$, it follows that \eqref{spherevariables} equals
\begin{multline*}
\bigg(\prod_{h=1}^{k-1}\prod_{m=1}^{h}\frac{\omega_{n-m}}{\omega_{n-m+1}}\bigg)
\int_{s_1=0}^{\infty}\cdots\int_{s_{\lambda}=0}^{\infty}\int_{\Omega}f\Big(s_1,\ldots,s_{\lambda},\sqrt n\big|\alpha_{12}-\sfrac{\pi}{2}\big|,\ldots,\sqrt n\big|\alpha_{(k-1)k}-\sfrac{\pi}{2}\big|\Big)\\
\times\prod_{1\leq i<j\leq k}\Big(\sin^{n-k-1}(\phi_{ij})\sin(\alpha_{ij})\Big)\,d\alpha_{(k-1)k}\ldots d\alpha_{12}ds_{\lambda}\ldots ds_1.
\end{multline*}
Recall that $\text{supp} f\subset[0,K]^{\lambda+\binom{k}{2}}$. Set $\vecp=(\frac{\pi}{2},\ldots,\frac{\pi}{2})\in\R^{\binom{k}{2}}$ and note that $J(\vecp)=\vecp$; hence $\Omega$ contains a neighbourhood of $\vecp$. In particular, for $n$ sufficiently large, we have $\vecp+n^{-\frac{1}{2}}[-K,K]^{\binom{k}{2}}\subset\Omega$, and we thus get, with $\widetilde\alpha_{ij}:=\sqrt n(\alpha_{ij}-\frac{\pi}{2})$,
\begin{multline*}
=\big(n^{-\frac{1}{2}}\big)^{\binom{k}{2}}\bigg(\prod_{h=1}^{k-1}\prod_{m=1}^{h}\frac{\omega_{n-m}}{\omega_{n-m+1}}\bigg)
\int_{s_1=0}^{\infty}\cdots\int_{s_{\lambda}=0}^{\infty}\nonumber\\
\times\int_{\widetilde\alpha_{12}=-\infty}^{\infty}\cdots\int_{\widetilde\alpha_{(k-1)k}=-\infty}^{\infty}
f\big(s_1,\ldots,s_{\lambda},|\widetilde\alpha_{12}|,\ldots,|\widetilde\alpha_{(k-1)k}|\big)\\
\times\prod_{1\leq i<j\leq k}\Big(\sin^{n-k-1}(\phi_{ij})\cos\Big(\frac{\widetilde\alpha_{ij}}{\sqrt n}\Big)\Big)\,d\widetilde\alpha_{(k-1)k}\ldots d\widetilde\alpha_{12}ds_{\lambda}\ldots ds_1.\nonumber
\end{multline*} 
In this expression we, of course, understand that $(\phi_{12},\ldots,\phi_{(k-1)k})=J^{-1}\big(\vecp+n^{-\frac{1}{2}}(\widetilde\alpha_{12},\ldots,\widetilde\alpha_{(k-1)k})\big)$ when $\vecp+n^{-\frac{1}{2}}(\widetilde\alpha_{12},\ldots,\widetilde\alpha_{(k-1)k})\in\Omega$, and we may leave $(\phi_{12},\ldots,\phi_{(k-1)k})$ undefined for all other $(\widetilde\alpha_{12},\ldots,\widetilde\alpha_{(k-1)k})$, since there we anyway have $f\big(s_1,\ldots,s_{\lambda},|\widetilde\alpha_{12}|,\ldots,|\widetilde\alpha_{(k-1)k}|\big)=0$. As in the discussion below \eqref{i<j} we have, for any fixed $(\widetilde\alpha_{12},\ldots,\widetilde\alpha_{(k-1)k})$,
\begin{align*}
\sin^{n-k-1}(\phi_{ij})=\Big(1-\sfrac{1}{2}n^{-1}\widetilde\alpha_{ij}^2+O\big(n^{-\frac{3}{2}}\big)\Big)^{n-k-1}\to e^{-\frac{1}{2}\widetilde\alpha_{ij}^2}
\end{align*}
as $n\to\infty$. Hence, using again the fact that $f$ has compact support, combined with \eqref{superquotient} and the dominated convergence theorem, we conclude that
\begin{align}\label{formula}
&\mathbb E\Big(\sum_{(n_1,\ldots,n_{\lambda})\in M_{\lambda}} f\big(\mathcal{V}_{n_1},\ldots,\mathcal{V}_{n_{\lambda}},\widetilde\varphi_{n_1n_2},\ldots,\widetilde\varphi_{n_{k-1}n_k}\big)\Big)\\
&\to2^{-\lambda}\Big(\frac{2}{\pi}\Big)^{\frac{1}{2}\binom{k}{2}}\int_{s_1=0}^{\infty}\cdots\int_{s_{\lambda}=0}^{\infty}\int_{\eta_{12}=0}^{\infty}\cdots\int_{\eta_{(k-1)k}=0}^{\infty}
f\big(s_1,\ldots,s_{\lambda},\eta_{12},\ldots,\eta_{(k-1)k}\big)\nonumber\\
&\times\prod_{1\leq i<j\leq k}e^{-\frac{{\eta_{ij}}^2}{2}}\,d\eta_{(k-1)k}\ldots d\eta_{12}ds_{\lambda}\ldots ds_1\nonumber
\end{align}
as $n\to\infty$.

Finally, we show that the right hand side of \eqref{expect} actually equals the right hand side of \eqref{formula}. We have
\begin{multline}\label{Kingman}
\mathbb E\Big(\sum_{(n_1,\ldots,n_{\lambda})\in M_{\lambda}} f\big(T_{n_1},\ldots,T_{n_{\lambda}},\Phi_{n_1n_2},\ldots,\Phi_{n_{k-1}n_k}\big)\Big)\\
=\mathbb E\Big(\sum_{(n_1,\ldots,n_{\lambda})\in M_{\lambda}} F\big(T_{n_1},\ldots,T_{n_{\lambda}}\big)\Big),
\end{multline}
where $F(x_1,\ldots,x_{\lambda})$ equals
\begin{multline}\label{F}
\Big(\frac{2}{\pi}\Big)^{\frac{1}{2}\binom{k}{2}}\int_{t_{12}=0}^{\infty}\cdots\int_{t_{(k-1)k}=0}^{\infty}f\big(x_1,\ldots,x_{\lambda},t_{12},\ldots,t_{(k-1)k}\big)\\
\times\prod_{1\leq i<j\leq k}e^{-\frac{{t_{ij}}^2}{2}}\,dt_{(k-1)k}\ldots dt_{12}.
\end{multline} 
Now, since $f$ in fact belongs to $L^1\big((\R_{\geq0})^{\lambda+\binom{k}{2}}\big)$, it is clear that the function $F$ defined by \eqref{F} belongs to $L^1\big((\R_{\geq0})^{\lambda}\big)$. Hence, by Campbell's Theorem (cf.\ Lemma \ref{lastlemma} below), the expectation value in \eqref{Kingman} (i.e.\ the right hand side of \eqref{expect}) equals the right hand side of  \eqref{formula}. This concludes the proof of the theorem.
\end{proof}

As we are not aware of a reference giving the precise version of Campbell's Theorem needed above, we state it here as a lemma.

\begin{lem}\label{lastlemma}
Let $k\in\Z_{\geq2}$. Then 
\begin{align}\label{functional}
\mathbb E\Big(\sum_{(n_1,\ldots,n_{k})\in M_{k}} f\big(T_{n_1},\ldots,T_{n_{k}}\big)\Big)=2^{-k}\int_{(\R_{\geq0})^k}f(x_1,\ldots,x_k)\,dx_1\ldots dx_k
\end{align}
for all $f\in L^1\big((\R_{\geq0})^k\big)$. 
\end{lem}

\begin{remark}\label{lastremark}
It follows from Campbell's Theorem \cite[p.\ 28]{king} (cf.\ also \cite[eq.\ (9)]{jag}) that the identity \eqref{functional} holds for any $f\in L^1\big((\R_{\geq0})^{k}\big)$ which can be factorized as $f(x_1,\ldots,x_{k})=\prod_{j=1}^{k}f_j(x_j)$ with some functions $f_j\in L^1(\R_{\geq0})$.
\end{remark}

\begin{proof}[Proof of Lemma \ref{lastlemma}]
By basic measure theory 
it is sufficient to prove \eqref{functional} for characteristic functions $f=\chi_A$ with $A\subset(\R_{\geq0})^k$ a measurable set  of finite measure. However, this follows from Remark \ref{lastremark} using standard techniques in the theory of product measures, as in the proof of \cite[Thm.\ 8.6]{rudin}.
\end{proof}

\section{Proof of Theorem \ref{poissongauss}}\label{proofsec}

We begin this section with a lemma dealing with the probability of repetitions in the sequence $\{\mathcal V_j\}_{j=1}^{\infty}$.

\begin{lem}\label{measurelemma}
For each $n\in \Z_{\geq1}$
\begin{align*}
\text{Prob}_{\mu_n}\Big\{L\in X_n \,\,\big|\,\, \mathcal V_1<\mathcal V_2<\mathcal V_3<\ldots\Big\}=1.
\end{align*}
\end{lem}

\begin{proof}
For $n=1$ the result holds trivially. For every given $n\geq2$ we have
\begin{align*}
&\mu_n\Big(X_n\setminus\big\{L\in X_n \mid \mathcal V_1<\mathcal V_2<\mathcal V_3<\ldots\big\}\Big)\\
&\leq\mu_n\Big(\big\{M\in\SL(n,\R) \mid \exists\,\vecm_1\neq\pm\vecm_2\in\Z^n\setminus\{\vec0\} : \|\vecm_1M\|=\|\vecm_2M\|\big\}\Big)\\
&\leq\sum_{\vecm_1\in\Z^n\setminus\{\vec0\}}\sum_{\vecm_2\in\Z^n\setminus\{\vec0,\pm\vecm_1\}}\mu_n\Big(\big\{M\in\SL(n,\R) \mid  \|\vecm_1M\|=\|\vecm_2M\|\big\}\Big).
\end{align*}
Since there are only countably many possible pairs $(\vecm_1,\vecm_2)$ it suffices to prove that each term in the last sum vanishes. By the explicit formula for the Haar measure on $\SL(n,\R)$ in terms of the matrix entries (cf.\ \cite[p.\ 7, Ex.\ 3 and p.\ 23, Ex.\ 23]{terras}), we see that it is enough to prove that, for any $\vecm_1,\vecm_2\in\Z^n\setminus\{\vec0\}$ with $\vecm_1\neq\pm\vecm_2$, 
\begin{align*}
\lambda_n\Big(\big\{M\in\mathrm{Mat}_{n,n}(\R) \mid  \|\vecm_1M\|^2-\|\vecm_2M\|^2=0\big\}\Big)=0,
\end{align*}
where $\lambda_n$ is the Lebesgue measure on $\mathrm{Mat}_{n,n}(\R)\cong\R^{n^2}$. However, this follows since $M\mapsto\|\vecm_1M\|^2-\|\vecm_2M\|^2$ is a non-zero homogeneous quadratic polynomial in the coefficients of $M$. 
\end{proof}

It follows from e.g.\ \cite[Thm.\ 5.3]{BW} that Theorem \ref{poissongauss} can be stated in the following equivalent form:

\begin{thm}\label{poissongauss2}
Let $N\in\Z_{\geq2}$.  Then, for all $f\in C\big((\R_{\geq0})^{N+\binom{N}{2}}\big)$ with compact support,
\begin{align*}
&\mathbb E\Big(f\big(\mathcal{V}_{1},\ldots,\mathcal{V}_{N},\widetilde\varphi_{12},\ldots,\widetilde\varphi_{(N-1)N}\big)\Big)\to\mathbb E\Big(f\big(T_{1},\ldots,T_{N},\Phi_{12},\ldots,\Phi_{(N-1)N}\big)\Big) 
\end{align*}
as $n\to\infty$. 
\end{thm}

\begin{proof}[Proof of Theorem \ref{poissongauss2}]
Let us fix $f\in C\big((\R_{\geq0})^{N+\binom{N}{2}}\big)$ with compact support. Note that without loss of generality we can assume that $f$ is non-negative. We further note that the identity
\begin{align}\label{first}
&f\big(\mathcal{V}_{1},\ldots,\mathcal{V}_{N},\widetilde\varphi_{12},\ldots,\widetilde\varphi_{(N-1)N}\big)\nonumber\\
&=\sum_{(n_1,\ldots,n_N)\in M_N} f\big(\mathcal{V}_{n_1},\ldots,\mathcal{V}_{n_N},\widetilde\varphi_{n_1n_2},\ldots,\widetilde\varphi_{n_{N-1}n_N}\big)\\
&\times I\big(\mathcal{V}_{n_1}<\ldots<\mathcal{V}_{n_N} \,\text{ and }\, \mathcal{V}_{j}\geq \mathcal{V}_{n_N} \text{ for } j\notin\{n_1,\ldots,n_{N-1}\}\big)\nonumber
\end{align}
holds for all $L\in Z_n$, where 
\begin{align*}
Z_n:=\big\{L\in X_n \mid  \mathcal V_1<\mathcal V_2<\mathcal V_3<\ldots\big\}.
\end{align*}
We recall from Lemma \ref{measurelemma} that $\mu_n(Z_n)=1$ for $n\geq1$.  

Since the sum in \eqref{first} is not of the form in Theorem \ref{rogersthm} we approximate it by combinations of sums that we can handle. In particular we introduce, for each $\ell\geq0$, the random variables 
\begin{align*}
R_{\ell}^n(L):=&\sum_{(n_1,\ldots,n_{N+\ell})\in M_{N+\ell}} f\big(\mathcal{V}_{n_1},\ldots,\mathcal{V}_{n_N},\widetilde\varphi_{n_1n_2},\ldots,\widetilde\varphi_{n_{N-1}n_N}\big)\\
&\times I\big(\mathcal{V}_{n_1}<\ldots<\mathcal{V}_{n_N} \,\text{ and }\,\mathcal{V}_{n_{N+1}}<\mathcal{V}_{n_{N+2}}<\ldots<\mathcal{V}_{n_{N+\ell}}<\mathcal{V}_{n_N} \big)\nonumber,
\end{align*}
and 
\begin{align*}
S_{\ell}^n(L):=\sum_{j=0}^{\ell}(-1)^{j}R_{j}^n(L).
\end{align*}
(The sum in the definition of $R_{\ell}^n(L)$ is finite with probability $1$, since $f$ has compact support. Thus $R_{\ell}^n(L)$ is well-defined.) We also introduce the corresponding random variables expressed in terms of the expected limit variables; that is, for $\ell\geq0$, we let
\begin{align*}
R_{\ell}^{\infty}:=&\sum_{(n_1,\ldots,n_{N+\ell})\in M_{N+\ell}} f\big(T_{n_1},\ldots,T_{n_N},\Phi_{n_1n_2},\ldots,\Phi_{n_{N-1}n_N}\big)\\
&\times I\big(T_{n_1}<\ldots<T_{n_N} \,\text{ and }\,T_{n_{N+1}}<T_{n_{N+2}}<\ldots<T_{n_{N+\ell}}<T_{n_N} \big)\nonumber,
\end{align*}
and
\begin{align*}
S_{\ell}^{\infty}:=\sum_{j=0}^{\ell}(-1)^{j}R_{j}^{\infty}.
\end{align*}
We note that it follows from Theorem \ref{rogersthm}, with $k=N$, that
\begin{align*}
\lim_{n\to\infty}\mathbb E\big(R_{\ell}^n(\cdot)\big)=\mathbb E\big(R_{\ell}^{\infty}\big)
\end{align*}
and hence also that
\begin{align}\label{Tlimit}
\lim_{n\to\infty}\mathbb E\big(S_{\ell}^n(\cdot)\big)=\mathbb E\big(S_{\ell}^{\infty}\big)
\end{align}
for all $\ell\geq0$.

Now let
\begin{align*}
N_n(L,x):=\#\{j:\mathcal V_j\leq x\}
\end{align*}
and 
\begin{align*}
N_{\infty}(x):=\#\{j:T_j\leq x\}.
\end{align*}
Using  $N_n(L,x)$ we can rewrite $R_{\ell}^n(L)$ for all $L\in Z_n$ as
\begin{align*}
R_{\ell}^n(L)&=\sum_{(n_1,\ldots,n_N)\in M_N}\binom{N_n(L,\mathcal V_{n_N})-N}{\ell}  I\big(\mathcal{V}_{n_1}<\ldots<\mathcal{V}_{n_N}\big)\\
&\times f\big(\mathcal{V}_{n_1},\ldots,\mathcal{V}_{n_N},\widetilde\varphi_{n_1n_2},\ldots,\widetilde\varphi_{n_{N-1}n_N}\big).
\end{align*}
It follows that for $L\in Z_n$ we also have the identity
\begin{align}\label{Tidentity}
S_{\ell}^{n}(L)&=\sum_{(n_1,\ldots,n_N)\in M_N}f\big(\mathcal{V}_{n_1},\ldots,\mathcal{V}_{n_N},\widetilde\varphi_{n_1n_2},\ldots,\widetilde\varphi_{n_{N-1}n_N}\big)\\
&\times I\big(\mathcal{V}_{n_1}<\ldots<\mathcal{V}_{n_N}\big)\sum_{j=0}^{\ell}(-1)^{j}\binom{N_n(L,\mathcal V_{n_N})-N}{j} .\nonumber
\end{align}
Using $N_{\infty}(x)$ we also get similar expressions for $R_{\ell}^{\infty}$ and $S_{\ell}^{\infty}$.

By elementary properties of Pascal's triangle we have, for $m,\ell\geq0$:
\begin{align*}
\sum_{j=0}^{\ell}(-1)^j\binom{m}{j}=\begin{cases}
1 & \text{if $m=0$}\\
(-1)^{\ell}\binom{m-1}{\ell} & \text{if $m>\ell$} \\
0 & \text{if $\ell\geq m>0$.}
\end{cases}
\end{align*}
From this and the relations \eqref{first} and \eqref{Tidentity} we find that for \textit{even} $\ell$,
\begin{align*}
S_{\ell}^{n}(L)&\geq\sum_{(n_1,\ldots,n_N)\in M_N}f\big(\mathcal{V}_{n_1},\ldots,\mathcal{V}_{n_N},\widetilde\varphi_{n_1n_2},\ldots,\widetilde\varphi_{n_{N-1}n_N}\big)\nonumber\\
&\times  I\big(\mathcal{V}_{n_1}<\ldots<\mathcal{V}_{n_N} \,\text{ and }\, N_n(L,\mathcal V_{n_N})=N\big) \\
&=f\big(\mathcal{V}_{1},\ldots,\mathcal{V}_{N},\widetilde\varphi_{12},\ldots,\widetilde\varphi_{(N-1)N}\big)
\end{align*}
for all $L\in Z_n$. Similarly, for \textit{odd} $\ell$ we have
\begin{align*}
S_{\ell}^{n}(L)&\leq\sum_{(n_1,\ldots,n_N)\in M_N}f\big(\mathcal{V}_{n_1},\ldots,\mathcal{V}_{n_N},\widetilde\varphi_{n_1n_2},\ldots,\widetilde\varphi_{n_{N-1}n_N}\big)\nonumber\\
&\times  I\big(\mathcal{V}_{n_1}<\ldots<\mathcal{V}_{n_N} \,\text{ and }\, N_n(L,\mathcal V_{n_N})=N\big) \\
&=f\big(\mathcal{V}_{1},\ldots,\mathcal{V}_{N},\widetilde\varphi_{12},\ldots,\widetilde\varphi_{(N-1)N}\big)
\end{align*}
for all $L\in Z_n$. Since $\mu_n(Z_n)=1$ for $n\geq1$ it follows that 
\begin{align}\label{limsup}
&\limsup_{n\to\infty}\mathbb E\Big(f\big(\mathcal{V}_{1},\ldots,\mathcal{V}_{N},\widetilde\varphi_{12},\ldots,\widetilde\varphi_{(N-1)N}\big)\Big)\leq\limsup_{n\to\infty}\mathbb E\big(S_{\ell}^{n}(\cdot)\big)=\mathbb E\big(S_{\ell}^{\infty}\big)
\end{align}
for all even $\ell$, and 
\begin{align}\label{liminf}
&\liminf_{n\to\infty}\mathbb E\Big(f\big(\mathcal{V}_{1},\ldots,\mathcal{V}_{N},\widetilde\varphi_{12},\ldots,\widetilde\varphi_{(N-1)N}\big)\Big)\geq\liminf_{n\to\infty}\mathbb E\big(S_{\ell}^{n}(\cdot)\big)=\mathbb E\big(S_{\ell}^{\infty}\big)
\end{align}
for all odd $\ell$.

To conclude the proof we determine the limit of $\mathbb E\big(S_{\ell}^{\infty}\big)$ as $\ell\to\infty$. We note that \eqref{Tidentity} holds almost surely when $n$ is replaced by $\infty$. As a consequence, using that $f$ has compact support, we find that 
\begin{align}\label{Tinf}
S_{\ell}^{\infty}&\to\sum_{(n_1,\ldots,n_N)\in M_N}f\big(T_{n_1},\ldots,T_{n_N},\Phi_{n_1n_2},\ldots,\Phi_{n_{N-1}n_N}\big)\nonumber\nonumber\\
&\times  I\big(T_{n_1}<\ldots<T_{n_N} \,\text{ and }\, N_{\infty}(T_{n_N})=N\big) \\
&=f\big(T_{1},\ldots,T_{N},\Phi_{12},\ldots,\Phi_{(N-1)N}\big)\nonumber
\end{align}
almost surely as $\ell\to\infty$. Since $f$ is a compactly supported continuous function there exist positive constants $K_1$ and $K_2$ such that $|f|\leq K_1$ and $\text{supp} f\subset\Big\{\vecx\in(\R_{\geq0})^{N+\binom{N}{2}}\,\big|\, |\vecx|\leq K_2\Big\}$. Hence, using \eqref{Tidentity} and the binomial theorem, we find that for each $\ell$, we have almost surely
\begin{align}\label{Test}
|S_{\ell}^{\infty}|&\leq K_1\sum_{(n_1,\ldots,n_N)\in M_N} I\big(T_{n_1}<\ldots<T_{n_N}\leq K_2\big)2^{N_{\infty}(T_{n_N})-N}\\
&\leq K_1\cdot N_{\infty}(K_2)^N\cdot2^{N_{\infty}(K_2)}\ll3^{N_{\infty}(K_2)}\nonumber
\end{align}
(where the implied constant depends only on $K_1$ and $N$). Using that $N_{\infty}(K_2)$ is Poisson distributed with mean $\frac12K_2$ we furthermore find that
\begin{align}\label{Nest}
\mathbb E\big(3^{N_{\infty}(K_2)}\big)=e^{-K_2/2}\sum_{k=0}^{\infty}\frac{1}{k!}\Big(\frac{3K_2}{2}\Big)^k=e^{K_2}.
\end{align}
Now, by \eqref{Tinf}, \eqref{Test}, \eqref{Nest} and the dominated convergence theorem, we obtain
\begin{align}\label{almost}
\lim_{\ell\to\infty}\mathbb E\big(S_{\ell}^{\infty}\big)=\mathbb E\Big(f\big(T_{1},\ldots,T_{N},\Phi_{12},\ldots,\Phi_{(N-1)N}\big)\Big) .
\end{align}
Finally, it follows from \eqref{limsup}, \eqref{liminf} and \eqref{almost} that
\begin{align*}
&\mathbb E\Big(f\big(\mathcal{V}_{1},\ldots,\mathcal{V}_{N},\widetilde\varphi_{12},\ldots,\widetilde\varphi_{(N-1)N}\big)\Big)\to\mathbb E\Big(f\big(T_{1},\ldots,T_{N},\Phi_{12},\ldots,\Phi_{(N-1)N}\big)\Big) 
\end{align*}
as $n\to\infty$, which is the desired result.
\end{proof}

\section{Application to successive minima}\label{sucmin}

For $L\in X_n$ the $i$:th successive minimum of $L$, $1\leq i\leq n$, is defined by
\begin{align*}
\lambda_i(L):=\min\big\{\lambda\in\R_{\geq0} \mid \text{$L$ contains $i$ linearly independent vectors of length $\leq\lambda$}\big\}.
\end{align*} 
Equipped with Proposition \ref{concentrate} we are able to describe the behavior of successive minima of random lattices in large dimensions. 

\begin{thm}
Let $N\in\Z_{\geq1}$. Then 
\begin{align*}
\text{Prob}_{\mu_n}\Big\{L\in X_n \,\big|\, \lambda_i(L)=|\vecv_i|,\, 1\leq i \leq N\Big\}\to1
\end{align*}
as $n\to\infty$.
\end{thm}

\begin{proof}
We claim that given $N\in\Z_{\geq1}$ there exists an $\ve>0$ such that for all vectors $\vecw_1,\ldots,\vecw_N\in\R^n\setminus\{\vec0\}$ ($n\geq N$) satifying $\varphi(\vecw_i,\vecw_j)>\frac{\pi}{2}-\ve$ for all $i\neq j$, we have $\dim(\text{Span}\{\vecw_1,\ldots,\vecw_N\})=N$. Indeed, given $\vecw_1,\ldots,\vecw_N\in\R^n\setminus\{\vec0\}$, we can without loss of generality assume that $|\vecw_1|=\ldots=|\vecw_N|=1$. Recall that the Gram matrix $G$ of the vectors $\vecw_1,\ldots,\vecw_N$ is given by $G=(g_{ij})_{1\leq i,j\leq N}=(\vecw_i\cdot\vecw_j)_{1\leq i,j\leq N}$. We have $g_{11}=\ldots=g_{NN}=1$ and $|g_{ij}|<\sin\ve<\ve$ for all $i\neq j$. Hence it is clear that, for $\ve$ small enough, the matrix $G$ has a non-zero determinant. It follows that $\vecw_1,\ldots,\vecw_N$ are linearly independent, which proves our claim. Finally, this observation together with Proposition \ref{concentrate} gives the desired result.
\end{proof}

\begin{cor}\label{succmin}
For any fixed $N\in\Z_{\geq1}$, the $N$-dimensional random vector 
\begin{align*}
\big(V_n\lambda_1(\cdot)^n,\ldots,V_n\lambda_N(\cdot)^n\big)
\end{align*}
converges in distribution to the distribution of the first $N$ points of a Poisson process on the positive real line with intensity $\frac{1}{2}$ as $n\to\infty$.  
\end{cor}

\subsubsection*{Acknowledgement} I would like to thank Jens Marklof whose questions concerning the paper \cite{jag} were the starting point for the investigations presented here. I am indebted to Svante Janson for suggesting the approach to Theorem \ref{poissongauss}, and also for making me aware of Stam's paper \cite{stam}. I am grateful to my advisor Andreas Str\"ombergsson for helpful and inspiring discussions on this work. Finally I would like to thank Igor Wigman for asking about successive minima.

\end{document}